\documentclass[a4paper, 11pt]{article}

\usepackage{amsmath}
\usepackage{graphicx}
\usepackage{amssymb}
\usepackage{amsthm}
\usepackage{graphicx}
\usepackage{amscd,amssymb,euscript,graphics}
\usepackage{enumerate}

\newcommand{\eqn}[1]{(\ref{#1})}
\newcommand{\fig}[1]{Fig. \ref{#1}}

\newtheorem{theorem}{Theorem}
\newtheorem{corollary}{Corollary}

\newtheorem{lemma}{Lemma}

\begin{document}

\title{Central Limit for the Product of Free Random Variables}

\author{Keang-Po Ho\\e-mail: kpho@ieee.org}


\date{\today}

\maketitle

\begin{abstract}
The central limit for the product of free random variables are studied by evaluating all the moments of the limit distribution.
The logarithm of the central limit is found to be the same as the sum of two independent free random variables: one semicircularly distributed and another uniformly distributed. 
The logarithm of central limit has a moment-generating function of $\exp(\xi^2 s/2) {_{1}F_{1}}\left(1-s; 2; -\xi^2 s \right)$. 
\end{abstract}


\section{Introduction}
Free probability theory \cite{nica06,voiculescu92} is for free random variables. 
The theory for free random variables are applicable to study random matrices with a very large dimension \cite{speicher03} \cite[\S 4]{voiculescu92}.
The statistical properties of free random variables are equivalently that of the eigenvalues of large random matrices.

In free probability theory, the central limit theorem on the sum of independent free random variables gives semicircle distribution \cite{voiculescu86,voiculescu91} \cite[\S 3.5]{voiculescu92}. Semicircle distribution, may be called free normal distribution, for free random variables serves the same as normal distribution for regular commuting random variables.

For positive conventional commuting random variables, the product of positive random variables has a central limit as the log-normal distribution because of $x_1 x_2  = \exp(\log x_1 + \log x_2 )$ for positive random variables $x_1$ and $x_2$.
As the central limit for the logarithm of a random variable is normal distribution, the product of random variables has its central limit as log-normal distribution. 
Log-normal distribution is applicable to many physical phenomena \cite{limpert01}, the shadowing of wireless channel \cite[\S 4.9.2]{rappaport02}, and some nonlinear noise in fiber communications \cite{ho0007}.
Log-normal distribution is also very close to the power law or Zipf law \cite{mitzenmacher04,newman05}.

For positive free random variables, the relationship of $X_1 \boxtimes X_2$ cannot guarantee the same as $\exp(\log X_1 \boxplus \log X_2)$, the symbols of $\boxtimes$ and $\boxplus$ are free product and free addition of free random variables, respectively, and equivalent to matrix multiplication and addition, respectively  \cite{speicher03} \cite[\S 4]{voiculescu92}.
Alternatively, the relationship of $\exp(X_1 \boxplus X_2)$ is not necessary equal to the free multiplication of $\exp(X_1) \boxtimes \exp(X_2)$. 
The log-semicircle distribution is not the central limit for the product of positive free random variables. 

Here, the central limit distribution is derived for the product of positive free random variables.
The logarithm of the central limit is found to be the sum of two independent free random variables: one uniformly distributed and another semicircularly distributed.
The moment-generating function of logarithm of the central limit is also equal to $\exp(\xi^2 s/2) {_{1}F_{1}}\left(1-s; 2; -\xi^2 s \right)$, where ${_{1}F_{1}}\left(a; b; x\right)$ is the confluent hypergeometric function. 

There was a long history to study the product of random matrices \cite{furstenberg60} to represent a physical system with the cascade of a chain of  basic random elements with identical statistical properties.
The results here are based on the multiplication of free random variables  \cite{voiculescu87}.
The characteristic of the central limit were derived in \cite{bercovici92}, together with all moments \cite{biane95}.
The central limit of the product of free random variables was studied in \cite{bercovici00,bercovici08,chistyakov08a} as infinitely divisible free random variables.
With all moments, the distribution is typically uniquely determined \cite{shohat} but difficult to find the distribution explicitly.
When the moment-generating function is translated to characteristic function, the distribution is just the inverse Fourier transform.  

To simplify the problem, for independent identically distributed positive free random variables $X_i$ with a product of $Y = \prod_{i=1}^n X_i$, we always assume that $\log X_i$ is zero mean. 
In \cite{kargin07}, the norm was considered for the case with unity mean finite support $X_i$.
The geometric mean was found in \cite{tucci10}, similar to the large number theory for commutating random variables.
The theory in \cite{bercovici00,bercovici08,chistyakov08a,kargin07,tucci10} may applicable to the case that $\log X_i$ do not exist, i.e., with an atom in zero. 
The theory here is only applicable to the case with the existing of $\log X_i$.

For zero-mean $\log X_i$, the central limit is only meaningful when the accumulated variance, $\xi^2 = \sum_{i=1}^n \varphi \left(\log^2 X_i  \right)$, is not very large, where $\varphi(A^k)$ denotes the $k$th-moment of a free random variable $A$.
The geometric mean  \cite{tucci10} or the Lyapunov exponent \cite{furstenberg60}, is very meaningful when the accumulated variance is far larger than unity, $\xi^2 >> 1$. 
In practical applicable, both geometric mean or Lyapunov exponent find the average contribution of each component $X_i$, especially when the number $n$ approaches infinity such that $\log Y$ has infinite variance. 
In some applications, $\log X_i$ may have a variance very close to zero but $\xi^2$ is finite with an order close to unity.
In engineering applications requiring meaningful approximation, the central limit may be used when the number $n$ is finite but large to render the validity of central limit. 
Here, we are interested when $\xi$ is at almost 10 to 20 dB.  

Later parts of this paper are organized as following: Sect. \ref{sec:main} summarizes the main results of this paper; Sects. \ref{sec:mom} and \ref{sec:times} give the proof of the theorem related to the product of free random variables; Sect. \ref{sec:num} shows some properties of the central limit and plots the distribution; Sect. \ref{sec:con} is the conclusion of this paper.

\section{Main Results}
\label{sec:main}

The goal here is to find the central limit for the product of independent free random variables.
The moments of the central limit are found explicitly. 
The moment-generating function is derived analytically by the method of matching all order of the moments. 
To keep $X_i$ and the final product $Y$ to have ``similar'' magnitude in variance, each component of the product is scaled by $X_i^{{1}/{\sqrt{n}}}$.
The main theorem of this paper is:

\begin{theorem}
Let $X_i$, $i = 1, \dots, n$,  denote identical independently distributed positive free random variables, if  $\log X_i$ are zero mean and have variance of $\xi^2 = \varphi\left( \log^2 X_i \right)$, the central limit of 
\begin{equation}
Y = \lim_{n \rightarrow \infty}  X_1^{\frac{1}{\sqrt{n}}} \boxtimes  X_2^{\frac{1}{\sqrt{n}}} \boxtimes \cdots \boxtimes  X_n^{\frac{1}{\sqrt{n}}}
\label{eq:ydef}
\end{equation}
has the following properties:
\begin{enumerate}[(a)]
\item The central limit is 

\begin{equation}
\log Y = \xi G + \frac{1}{2} \xi^2 F 
\label{logYsum}
\end{equation}

\noindent where $G$ is a free-random variable with zero-mean semicircle distribution and unity variance, $F$ is a free random variable uniformly distributed between $\pm1$, and $G$ and $F$ are independent of each other. 

\item The moment-generating function of $\log Y$ is 

\begin{equation}
\exp(\xi^2 s/2) {_{1}F_{1}}\left(1-s; 2; -\xi^2 s \right).
\label{logYmfun}
\end{equation} 

\end{enumerate}
\label{the1}
\end{theorem}

If Theorem \ref{the1} is correct, we should have the following straightforward corollaries:

\begin{corollary}
The moment-generating function of 

\begin{equation}
S = \xi G + \frac{1}{2} \xi^2 F
\label{Sdef}
\end{equation}
\noindent is
\begin{equation}
M_S(s) = \exp(\xi^2 s/2) {_{1}F_{1}}\left(1-s; 2; -\xi^2 s \right).
\label{logSmfun}
\end{equation} 
\label{cor2}
\end{corollary}

\begin{corollary}
The $k$th-moment of the central limit $Y$ \eqn{eq:ydef} is
\begin{equation}
\varphi\left(Y^k\right) = \exp\left(k \xi^2/2\right) {_{1}F_{1}}(1-k; 2; -k \xi^2).
\label{kmomentY}
\end{equation}
\label{cor1}
\end{corollary}

%
%

This study is always limited to the case with zero mean $\log X_i$, leading to zero-mean $\log Y$ as from \eqn{logYsum}.
If $\log X_i$ have positive mean, $Y$ \eqn{eq:ydef} grows to infinity. 
If $\log X_i$ have negative mean, $Y$ \eqn{eq:ydef} shrinks to zero.
Only the cases with zero mean $\log X_i$ is interested for the study of $Y$ \eqn{eq:ydef}.
In practice, the mean of $\log Y$ may be studied separately as the Lyapunov exponent as in \cite{furstenberg60}. 

If Corollaries \ref{cor2} and \ref{cor1} are both correct, Theorem \ref{the1} can be proved with minimal efforts. 
With moment-generating function of $M_S(s)$ of \eqn{logSmfun} and by the definition of moment-generating function, the $k$th-moment of $\exp(S)$ is $M_S(k)$  that is the same as \eqn{kmomentY}.
Because the $k$th-moment of $Y$ is the same as the $k$th-moment of $\exp(S)$, $Y$ has the same distribution as $\exp(S)$ as the moment uniquely determined the distribution \cite{shohat}.
Because $Y$ is positive definite, $\log Y$ has the same distribution as $S$ given by \eqn{logYsum} or \eqn{Sdef}.
Because of Corollary \ref{cor2}, $\log Y$ also has moment-generating function given by \eqn{logYmfun} or \eqn{logSmfun}. 
Note that the uniqueness requires the moment to satisfy the Carleman condition \cite[\S 1.6]{shohat} that is easy to verify. 

The following sections prove Corollaries \ref{cor2} and \ref{cor1}, respectively.

\section{moment-generating function for $S = \xi G + \frac{1}{2} \xi^2 F$}
\label{sec:mom}

The propose of this section is to find the moment-generating function of $S = \xi G + \frac{1}{2} \xi^2 F$ \eqn{Sdef} as given by Corollary \ref{cor2}.
We will first derive the $n$th-moment of $S$ and compared it with the $n$th-moment given by the moment-generating function $M_S(s)$ \eqn{logSmfun}.

The statistics of the sum of free random variables can be found by the sum of their corresponding $R$-transform \cite{nica06, voiculescu92}.
The $R$-transform is defined by the algebraic relationship of 

\begin{equation}
G\left( R(z) + \frac{1}{z} \right) = z
\label{Rdef}
\end{equation}
where $G(z)$ is the Cauchy-Stieltjes transform of a measure $A$ that is given by the expectation of

\begin{equation}
G(z) = \int \frac{1}{z - t} d \mu_A(t) 
 = \sum_{k = 0}^{\infty} \frac{\varphi\left(A^k\right)}{z^{k+1}}.
 \label{Gdef}
 \end{equation}
 Here, the $k$th-moment of a free random variable is denoted as $\varphi\left(A^k\right)$.
 If $A$ is a large random matrix, $\varphi\left(A^k\right)$ is the $k$th-moment of the eigenvalues for $A$. 
 
 In the summation \eqn{Sdef}, $G$ is semicircularly distributed with unity variance, or a radius of $2$.
 Typically, $G$ is assumed as a large random Gaussian matrix \cite{ho14} but many other large random matrices are possible \cite{ erdos10, tao09}.
 The $R$-transform for $G$ is $R_G(s) = s$ \cite{nica06, voiculescu92}.
 $F$ is a free random variable uniformly between $\pm1$, the $k$th moment for $F$ is $\varphi\left(F^k\right) = 1/(k+1)$ for even $k$ and $0$ otherwise.   
 The Cauchy-Stieltjes transform of $F$ is 
 
 \begin{equation}
 G_F(z) = \sum_{k=0}^\infty \frac{1}{(2k+1)z^{2k+1}} =  \coth^{-1} z
 \end{equation}
 \noindent and the $R$-transform for $\frac{1}{2} \xi^2 F$ is $\frac{1}{2} \xi^2 \coth\left( \frac{1}{2} \xi^2 z  \right) - 1/z$. 
 The $R$-transform for $S$ becomes
 
 \begin{equation}
 R_S(z) = \xi^2 z + \frac{1}{2} \xi^2 \coth\left( \frac{1}{2} \xi^2 z  \right) - \frac{1}{z}
 \end{equation}
 The Cauchy-Stieltjes transform of $S$ is given by the inverse function of 
 \begin{equation}
 G^{-1}_S(z) = \xi^2 z + \frac{1}{2} \xi^2 \coth\left( \frac{1}{2} \xi^2 z  \right)
 \end{equation}
\noindent From \eqn{Gdef}, instead of using $G(z)$, the moments of $S$ can be found more conveniently using $G(1/z)$. 

Using the Lagrange inversion formula \cite[\S 7.32]{whittaker} with variation in \cite{kargin07a,kargin07}, 
if the dependence between $w$ and $z$ is given by $w/\phi(w) = z$ and $\phi(0) \neq 0$, the inversion series for $w = g(z)$ is given by
\begin{equation}
g(z) = \sum_{k=1}^\infty \frac{z^k}{k!} \left.  \frac{d^{k-1}}{d w^{k-1}} \phi(w)^k \right|_{w = 0}.
\label{lagbur}
\end{equation}

Using the formulation of \eqn{lagbur}, the inverse of $G(1/z)$ is given by the inverse of $w/\phi(w)$ with 

\begin{equation}
\phi(w) = \xi^2 w^2 + 1 + \frac{1}{2} \xi^2 w H \left( \frac{1}{2}  \xi^2 w \right)
\label{phiwHdef}
\end{equation}
where $H(x) = \coth x - 1/x$.

The $n$-th moments of $S$ is given by
\begin{equation}
\left. \frac{1}{(n+1)!}\frac{d ^n}{ d w^n} \phi(w)^{n+1} \right|_{w = 0}
\end{equation}
\noindent from Lagrange inversion formula \eqn{lagbur}, or $1/(n+1)$ the coefficient of $w^n$ in $\phi(w)^{n+1}$.

With the expansion of 

\begin{equation}
\phi(w)^{n+1} = \sum_{k=0}^{n+1} \binom{n+1}{k}  \left[ \frac{1}{2} \xi^2 w H \left( \frac{1}{2}  \xi^2 w \right)\right]^{k}  \left(1 + \xi^2 w^2\right)^{n+1-k}
\label{phi_nplus1}
\end{equation} 

\noindent
For the coefficient at $w^n$, the first term of \eqn{phi_nplus1} may give $w^k$ for $k = 0$ to $n$, and the second term of \eqn{phi_nplus1} can give
\begin{equation}
\left.
\frac{\xi^{2i}} {i! 2^i} 
\frac{d^i}{d w^i} H^k(w) \right|_{w = 0} w^i
\end{equation}
\noindent and the third term of \eqn{phi_nplus1} gives $w^{n-k-i}$ with $n-k-i$ an even number.
Combined them together, the $n$th moment of $S$ becomes

\begin{equation}
\frac{1}{n+1} \sum_{k=0}^{n} \binom{n+1}{k} \frac{\xi^{2k}}{2^k}
\sum_{i=0}^{n-k} \frac{\xi^{2i}}{2 ^{i}i!}  \left. \frac{d^i}{d w^i} H^k(w) \right|_{w = 0} 
       \binom{n+1-k}{(n - k -i)/2} \xi^{n-k-i}
       \nonumber
\end{equation}
\noindent or
\begin{equation}
\frac{1}{n+1} \sum_{k=0}^{n} \binom{n+1}{k} 
\sum_{i=0}^{n-k} \frac{\xi^{n+k+i}}{2 ^{k+i} i!} \left. \frac{d^i}{d w^i} H^k(w) \right|_{w = 0} 
       \binom{n+1-k}{(n - k -i)/2} 
 \label{n_S_moment_1}
\end{equation}
\noindent with summation only when $n-k-i$ is even number. 

For arbitrary $i$ and $k$, the summation of \eqn{n_S_moment_1} cannot guarantee that $(n-k-i)/2$ is an integer. 
Because $H(x)$ defined for \eqn{phiwHdef} is an odd function, $k$ and $i$ must be both even and odd number together such that $\left. \frac{d^i}{d w^i} H^k(w) \right|_{w = 0}$ is non-zero. 
In order for $(n-k-i)/2$ as an integer, $n$ must be an even number.
Both $G$ and $F$ in \eqn{Sdef} have distribution symmetric with respect to zero, $S$ \eqn{Sdef} has only even moment that is the same as the requirement that $n-k-i$ is an even number.
 
If the $n$th-moment \eqn{n_S_moment_1} is also given by the moment-generating function of \eqn{logSmfun}, Corollary \ref{cor2} is correct.
Here, the $n$th moment of \eqn{logSmfun} is derived and found to be the same as \eqn{n_S_moment_1}.

First of all, the confluent hypergeometric function in \eqn{logSmfun} has the Buchholz expansion \cite[p. 97]{buchholz}

\begin{equation}
{}_1 F_1(1-a, 2; s)
    = 2 e^{-s/2} \sum_{k=0}^{\infty}
      p_k(s) \frac{J_{k+1}(2 \sqrt{a s})} {\left(2 \sqrt{a s}\right)^{k+1}}
\label{buchexp}
\end{equation}
 \noindent where $J_k(z)$ is the Bessel function of the first kind and $p_k(z)$ is the Buchholz polynomial given by \cite{abad95, lopez99}
 
\begin{equation}
\exp \left[ - \frac{1}{2} s \left( \coth w - \frac{1}{w} \right) \right]
  = \sum_{k = 0}^\infty p_k(s) \left( - \frac{w}{s} \right)^k
  \end{equation}
 \noindent or
 \begin{equation}
 p_k(s) = \frac{ (-z)^k}{k!} \left. \frac{d^k}{d w^k} \exp
 \left( - \frac{1}{2} s H(w) \right) \right|_{w = 0}
 \label{buchpol}
 \end{equation}  
 \noindent where $H(w) = \coth w - 1/w$ is the same as that defined for \eqn{phiwHdef}.
 Using the Buchholz expansion, the moment-generating function $M_S(s)$ \eqn{logSmfun} becomes
 \begin{equation}
 M_S(s) = 2 \sum_{k=0}^\infty
    \frac{\xi^{2k} s^k} {k!} 
    \left. \frac{d^k}{d w^k} \exp
 \left(\frac{1}{2} \xi^2 s H(w) \right) \right|_{w = 0}
    E_{k+1}(2 \xi s)
    \label{msexpand}
 \end{equation}
 \noindent where
 \begin{equation}
 E_k(x) = \frac{I_k(x)}{x^k}
   = 2^{-k} \sum_{m=0}^\infty 
     \frac{1}{m!(m+k)!} \left( \frac{x}{2} \right) ^{2m}
 \end{equation}
 where $I_k(x)$ is the modified Bessel function of the first kind.
 
 To find the $n$th-moment of $M_S(s)$ \eqn{logSmfun}, the expression $M_S(s)$ may be expanded in a power series of $s$ and the coefficient of $s^n$ is $n!$ times the $n$th-moment of the random variable given the the corresponding moment-generating function. 
 In the expansion of $M_S(s)$ \eqn{msexpand}, $s^k$ comes from each term of the expansion.
 There is also contribution from the third term of $E_{k+1}(2 \xi s)$ and from Buchholz polynomial [the term with $H(w)$].
 The three terms combined together give the coefficient for $s^n$.
 
 The coefficient of $s^i$  due to Buchholz polynomial is given by
 \begin{equation}
\left. \frac{1}{i!} \frac{d^i}{d s^i}
 \frac{d^k}{d w^k}
 \exp \left( \frac{1}{2} \xi^2 s H(w) \right) \right|_{w = s = 0}
   = \frac{\xi^{2i}}{2^i i!} 
  \left. \frac{d^k}{d w^k} H^i(w) \right|_{w = 0}
 \end{equation}
 
 With all terms and for $n -k -i$ as even number, the $n$th-moment from $M_S(s)$ is
 \begin{equation}
 2n! \sum_{k=0}^n
 \frac{\xi^{2k}} {k! 2^{k+1}}
 \sum_{i=0}^{n-k} \frac{\xi^{2i}}{2^i i!}
 \left. \frac{d^k}{d w^k} H^i(w) \right|_{w = 0}
 \frac{\xi^{n-k-i}}{[(n-k-i)/2]![(n+k-i)/2+1]!}
 \end{equation}
 \noindent or
 
 \begin{equation}
 \sum_{k=0}^n
 \frac{1} {k!}
 \sum_{i=0}^{n-k} 
 \left. \frac{d^k}{d w^k} H^i(w) \right|_{w = 0}
 \frac{\xi^{n+k+i} n!}{i! 2^{k+i} [(n-k-i)/2]![(n+k-i)/2+1]!}
 \label{n_S_moment_2}
 \end{equation}
 
 By swapping the indexes $i$ and $k$ in \eqn{n_S_moment_2}, the factors with $\xi$ and $H(w)$ in the $n$th-moment of \eqn{n_S_moment_1} are the same as that in \eqn{n_S_moment_2}.
 With little algebra and with the restriction that $n-k-i$ is even number, it may be found that 
 \begin{equation}
 \frac{1}{n+1}
 \binom{n+1}{k}
 \binom{n+1-k}{(n-k-i)/2}
 \end{equation}
 \noindent from \eqn{n_S_moment_1} is the same as
 \begin{equation}
 \frac{n!}{k![(n-k-i)/2]![(n+k-i)/2+1]!}
 \end{equation}
 \noindent from \eqn{n_S_moment_2}.
 
 The $n$th-moment of \eqn{n_S_moment_1} is the same as the $n$th-moment of \eqn{n_S_moment_2}.
 As the $n$th-moment of \eqn{n_S_moment_1} is derived based on free random variable theory of $S$ from \eqn{Sdef} and the $n$th-moment of \eqn{n_S_moment_2} is derived based on the moment-generating function \eqn{logSmfun}, Corollary \ref{cor2} is correct that the moment-generating function of $S$ from \eqn{Sdef} is $M_S(s)$ \eqn{logSmfun}.

\section{Products of Free Random Variables}
\label{sec:times}

Based on the $S$-transform \cite{rao07, voiculescu87}, the central limit for the product of free random variables can be derived. 
The moments for the central limit are derived here to prove Corollary \ref{cor1}.
The central limit for the product of free random variables was considered very early in \cite{bercovici92,  bercovici00, biane95} but a simple process is offered here. 

Instead of consider the product of positive independent identically distributed free random variables $X_i$, we assume that $v_i = \log X_i$ exist such that $\varphi\left(v_i\right) = 0$ and $\varphi\left(v_i^2\right) = 1$, $i = 1, \dots, n$.
The free random variables of $\exp(v_i)$, $i = 1, \dots, n$, are always positive.
The assumption of $\varphi\left(v_i^2 \right) = 1$ is not essential but only for convenience. 
Define
\begin{equation}
Y_n = \exp\left( \frac{v_1}{\sqrt{n}} \right) \boxtimes \exp\left( \frac{v_2}{\sqrt{n}}\right) \boxtimes \cdots \boxtimes \exp\left( \frac{v_n}{\sqrt{n}}\right),
\label{eq:yn}
\end{equation}
and 
\begin{equation}
Y = \lim_{n \rightarrow \infty} Y_n.
\label{eq:y}
\end{equation}
Same as \eqn{eq:ydef}, the free random variable $Y$ is considered here.

The product of free random variables can be studied based on the product of the corresponding $S$-transforms \cite{rao07, voiculescu87}.
The $S$-transform of a free random variable $A$ is given by
\begin{equation}
S(z) = \frac{1 + z}{z} \chi(z)
\label{StrDef}
\end{equation}
\noindent where $\chi(z)$ is the inverse function by composition, $\chi(\psi(z)) = z$, of the moment function  given by
\begin{equation}
\psi(z)  = \int \frac{zt}{1 - zt} d \mu_A(t)
   = \sum_{k=1}^\infty \varphi\left(A^k\right) z^k.
   \label{chidef}
\end{equation}

Using $S$-transform \eqn{StrDef} for $Y_n$ \eqn{eq:y}, $S_{Y_n}(z) = \left[ S_{e^{v_i/\sqrt{n}}}(z) \right]^n$ and $S_Y(z) = \lim_{n \rightarrow \infty}S_{Y_n}(z)$. 
Although $S$-transform \eqn{StrDef} may extend to free random variables that are not always positive \cite{rao07}, free random variables $\exp\left( {v_i}/{\sqrt{n}} \right)$ here are always positive. 

\begin{lemma} \cite{bercovici92}
The $S$-transform of $Y$  \eqn{eq:y} is
\begin{equation}
S_Y(z) = \exp\left( -z - \frac{1}{2}\right).
\end{equation}
\label{lemmasy}
\end{lemma}

\begin{proof}
We first derive $S_{e^{v_i/\sqrt{n}}}(z)$. 
The $S$-transform of $\exp\left( {v_i}/{\sqrt{n}} \right)$ can be found as a power series of $1/\sqrt{n}$, in essence the perturbation analysis of the $S$-transform.
Assume that $\psi_{v_i}(z) = z^2 + \sum_{k = 3}^{+\infty} \varphi\left(v_i^k\right) z^k$ as $v_i$ are zero mean, unity variance free random variables.
Note that the moments of $\varphi\left(v_i^k\right)$, $k \geq 3$, is not very important to arrive with the results but includes here for completeness.

The moment function of $\exp(v_i)$ is 
\begin{equation}
\psi_{e^{v_i}}(z) = \frac{z}{1-z} + \frac{z(1 + z)}{2 (1-z)^3} + \sum_{k \geq 3} \frac{\varphi\left(v_i^k\right)} {k!}\left. \frac{\partial^k}{\partial t^k} \left[\frac{z}{e^{-t} - z}\right] \right|_{t = 0},
\label{varphvi}
\end{equation}
where $\left. \frac{\partial^k}{\partial t^k} \left[\frac{z}{e^{-t} - z}\right] \right|_{t = 0} = \sum_{l \geq 1 } l^k z^l$ is the coefficient corresponding to $\varphi\left(v_i^k\right)$ for the moment function $\psi_{e^{v_i}}(z)$.
The summation of $\sum_{l \geq 1 } l^k z^l$ may be obtained by expanding $e^{k v_i}$ in $\sum_{k > 0} \varphi\left(e^{k v_i}\right) z^k$.

The moment function of $\exp\left( {v_i}/{\sqrt{n}} \right)$ as a power series of $1/\sqrt{n}$ is 
\begin{equation}
\psi_{e^{v_i/\sqrt{n}}}(z) = \frac{z}{1-z} \left[ 1 + \frac{1 + z}{2 n (1 - z)^2} 
			+ \varphi\left(v_i^3 \right) \frac{g_3(z)}{ 6 n^{\frac{3}{2}}} 
			- \varphi\left(v_i^4 \right) \frac{g_4(z)}{ 24 n^2} +O(n^{-\frac{5}{2}}) \right],
			\label{varphiv1}
\end{equation}
where 
\begin{eqnarray}
g_3(z) & = & \frac{1 + 4z + z^2}{  (1 - z)^3} \nonumber \\
g_4(z) & = & \frac{z^3 + 11z^2 + 11z + 1}{(1-z)^4} \nonumber.
\end{eqnarray}

The inverse function of \eqn{varphiv1} can also be expressed as a power series of $1/\sqrt{n}$ to
\begin{equation}
\chi_{e^{v_i/\sqrt{n}}}(z) = \frac{z}{ 1 + z}
    \left[1 - \frac{1}{n}\left( z + \frac{1}{2}\right) 
             - \frac{\varphi\left(v_i^3\right)}{n^{\frac{3}{2}}} h_3(z) 
             + \frac{\varphi\left(v_i^4\right)}{n^2} h_4(z)
              +  O(n^{-\frac{5}{2}})
     \right],
     \label{chiv1}
\end{equation}
where
\begin{eqnarray}
h_3(z) & = & z^2 + z + \frac{1}{6} \nonumber \\
h_4(z) & = & z^3 + 2 z^2 + \frac{7}{6}z+\frac{5}{24} \nonumber .
\end{eqnarray}
The $S$-transform of $e^{v_i/\sqrt{n}}$ is
\begin{equation}
S_{e^{v_i/\sqrt{n}}}(z) = 1 - \frac{1}{n}\left( z + \frac{1}{2}\right) 
             - \frac{\varphi\left(v_i^3\right)}{n^{\frac{3}{2}}} h_3(z) 
             + \frac{\varphi\left(v_i^4\right)}{n^2} h_4(z)
              +  O(n^{-\frac{5}{2}}).
\end{equation}

The $S$-transform of $Y_n$ \eqn{eq:yn} becomes a power series of $1/\sqrt{n}$ using the relationship of $S_{Y_n}(z) = [S_{e^{v_i/\sqrt{n}}}(z)]^n$  as
\begin{equation}
S_{Y_n}(z) 
= \exp\left\{ -z - \frac{1}{2} 
  - \frac{\varphi\left(v_i^3\right)}{\sqrt{n}} h_3(z) 
  + \frac{1}{n} \left[\varphi\left(v_i^4\right) h_4(z) - \frac{1}{2}\left(z + \frac{1}{2}\right)^2 \right]
   + O (n^{-\frac{3}{2}}) \right\}
   \label{synz}
\end{equation}
and
\begin{equation}
S_Y(z) = \lim_{n \rightarrow \infty} S_{Y_n}(z)
      = \exp\left( -z - \frac{1}{2}\right).
     \label{syz}
\end{equation}
%
\qed
\end{proof}

The $S$-transform of \eqn{syz} is infinitely divisible from the theory of \cite{bercovici00,bercovici08,chistyakov08a}.
In practice, all terms higher than the second-order is not required in \eqn{varphiv1}, \eqn{chiv1} and \eqn{synz}.
The $S$-transform of \eqn{synz} shows the convergent properties with the increase of $n$, similar to the analysis of \cite{kargin07a}.
For non-symmetric free random variable $v_i$ with non-zero skewness proportional to $\varphi\left(v_i^3\right)$, the convergence is far slower than the symmetric free random variable with zero skewness.

The inversion of the moment function $\psi_Y(z)$ is given by
\begin{equation}
\chi_Y(z) = \frac{z}{1 + z} \exp\left( -z - \frac{1}{2} \right).
\label{chiYz}
\end{equation}

Similar to \cite{biane95} and using the Lagrange inverse formula \eqn{lagbur},  we need to use $\phi(w) = (1+w)\exp\left( w + \frac{1}{2}\right)$.
From the definition of moment function \eqn{chidef}, the $k$th-moment of $Y$ is given by
\begin{equation}
\varphi\left(Y^k\right) = \frac{e^{k/2}}{k!} \left. \frac{d^{k-1}}{d w^{k-1}} (1 + w)^k e^{kw} \right|_{w = 0}
\end{equation}
or
\begin{equation}
\varphi\left(Y^k\right) = \frac{1}{k}\exp\left(\frac{k}{2}\right) \sum_{m=0}^{k-1} \frac{k^m}{m!} \binom{k}{m+1},
\label{kYmoment}
\end{equation}
by expanding $e^{kw}$ and collecting all terms with $w^{k-1}$ in $(1 + w)^k e^{kw}$.

Using the generalized Laguerre polynomial $L_n^{(\alpha)}(z)$ \cite{koornwinder10} \cite[\S 5.1]{szego75}, the $k$th-moment \eqn{kYmoment} can be expressed as
\begin{equation}
\varphi\left(Y^k\right) = \frac{e^{k/2}}{k} L^{(1)}_{k-1}(-k).
\label{kYmomentL}
\end{equation}

Confluent hypergeometric function ${_{1}F_{1}}(a; b; z)$ \cite[\S 5.3]{szego75} can also be used instead of the Laguerre polynomial.
The moment \eqn{kYmomentL} becomes
\begin{equation}
\varphi\left(Y^k\right) = e^{k/2} {_{1}F_{1}}(-k+1; 2; -k) = M_{k,\frac{1}{2}}(-k)/k,
\label{kYmomentF}
\end{equation}
where $M_{\kappa, \mu}(z)$ is the Whittaker function \cite[\S 16.1]{whittaker}.

With $\varphi\left(v_i^2\right) = \xi^2$, following the procedure from \eqn{varphvi} to \eqn{syz} without details, we get
\begin{equation}
S_Y(z) = \exp\left[-\xi^2 \left(z + \frac{1}{2}   \right)   \right].
\end{equation}

Following the procedure from \eqn{chiYz} to \eqn{kYmomentL} without details, we can obtain 
\begin{equation}
\varphi\left(Y^k\right) = \frac{e^{\xi^2 k/2}}{k} L^{(1)}_{k-1}(-\xi^2 k)
             = e^{k \xi^2/2} {_{1}F_{1}}(-k+1; 2; -k \xi^2)
\label{kYmomentLsigma}
\end{equation}
\noindent and Corollary \ref{cor1} is proved. 

If we have only the $k$th-moment $\varphi\left(Y^k\right)$  by itself, we cannot guarantee that the moment-generating function of $\log Y$ of $\varphi\left( e^{s \log Y} \right) = \varphi\left(Y^s\right)$ is the same as \eqn{logYmfun}.
The fractional moment given by \eqn{logYmfun} may not correspond to a random variable. 
However, if we are able to find a distribution having the same moment-generating function of \eqn{logYmfun}, due to the uniqueness of a distribution determined by the $k$th-moment, we can extend from $\varphi\left(Y^k\right)$  to $\varphi\left(Y^s\right)$.
From Corollary \ref{cor2}, the distribution is found to be the same as \eqn{logYsum}.

Combining both Corollaries \ref{cor2} and \ref{cor1}, Theorem \ref{the1} is proved. 

\section{Numerical Results}
\label{sec:num}

Using the results from Theorem \ref{the1}, this section gives some properties of the central limit for the product of free random variables. Some numerical results are also presented here.

Using the Kummer transformation \cite[\S 13.2.39]{olde10} with 
\begin{equation}
e^{\xi^2 s/2} {_{1}F_{1}}\left(1-s; 2; -\xi^2 s \right) = e^{-\xi^2 s/2} {_{1}F_{1}}\left(1+s; 2; \xi^2 s \right),
\end{equation}
the moment-generating function \eqn{logYmfun} is an even function with $\varphi(Y^s) = \varphi(Y^{-s})$.
All the odd moments of $\log Y$ are equal to zero. 
The free random variable $\log Y$ is symmetric with respect to zero.

Directly using Markov inequality \cite[\S 7.3.1]{grimmett}, we may obtain

\begin{equation}
\mathbf{Pr}\left(Y \geq r\right) = 0 ~~~\mathrm{if}~~~~ r >  \lim_{k \rightarrow \infty}\varphi(Y^k)^{1/k},
\label{eq:pospr}
\end{equation}
and the upper limit of a random variable is equal to $\lim_{k \rightarrow \infty}\varphi(Y^k)^{1/k}$.
From the asymptotic of Laguerre polynomial \cite[Eq. 4.7]{frenzen88}, 
we obtain
\begin{equation}
\varphi(Y^k) = 2 \alpha_0 \frac{I_1( c k)} {c k } [1 + O(k^{-1})],  
\label{ykasy}
\end{equation} 
where $\alpha_0$ is a constant and 
\begin{equation}
c = \xi \sqrt{1 + \frac{\xi^2}{4}} + 2 \log\left( \frac{\xi}{2} + \sqrt{1 + \frac{\xi^2}{4}} \right).
\label{c0main}
\end{equation}

Using the asymptotic of $I_1( z) \sim e^{z}/\sqrt{2 \pi z}$ \cite[\S 10.40.1]{olver10} for \eqn{ykasy}, and with \eqn{eq:pospr}, we obtain that
\begin{equation}
 \log Y \leq c
\end{equation}
Because $\log Y$ is symmetrical, we have $| \log Y | \leq c$. 
Similar results were derived in \cite{biane97}.

\fig{fig:pdf} shows the central limit distribution of $\log Y$ for various values of $\xi$.
The distribution is calculated by inverse Fourier transform of the characteristic function of $M_S(i\omega)$ with the moment-generating function $M_S(s)$ given by  \eqn{logYmfun}.
The $x$-axis is normalized by $c$ from \eqn{c0main} to limit the distribution to between $\pm 1$.
The value of $\xi$ is expressed in decibel scale that is $10/\log10 = 4.34$ times the value in typical linear scale. 

\begin{figure}[t]
\centering{
 \includegraphics[width = 0.9 \textwidth]{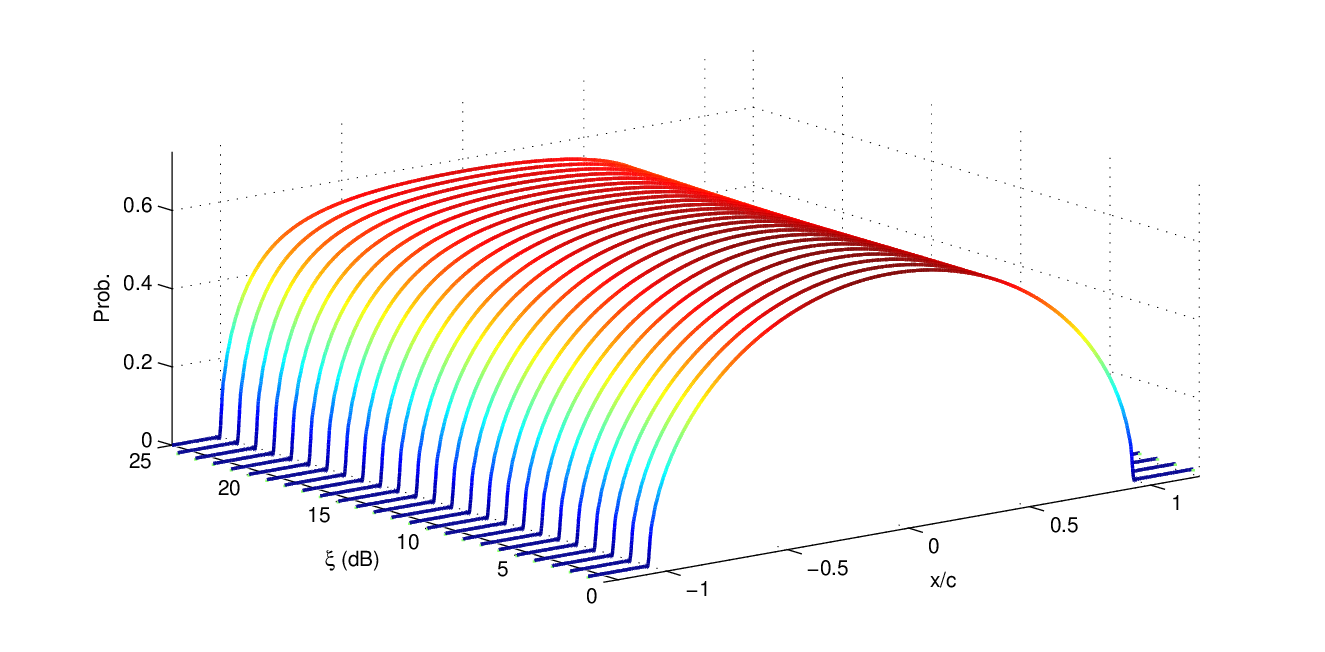}}
 \caption{The central limit distribution of $\log Y$ for different values of $\xi$.
  The $x$-axis is normalized by $c$ \eqn{c0main}.
 }
\label{fig:pdf}
\end{figure}

From \fig{fig:pdf}, the limited distribution is very close to semicircle distribution when $\xi$ is less than 10 to 15 dB.
The distribution is more close to uniform distribution when $\xi$ is larger than 20 dB.
From \eqn{logYsum}, the distribution is mainly from $G$ for small $\xi$, giving the semicircle distribution but the distribution is mainly from $F$ for large $\xi$, giving the uniform distribution.
For small $\xi$, the limit $c$ \eqn{c0main} is $c \approx 2 \xi$, and the distribution of $\log Y$ is a semicircle with radius of $2 \xi$. 
For large $\xi$, the limit $c$ \eqn{c0main} is $c \approx \xi^2/2$, and the distribution of $\log Y$ is uniform between $\pm \xi^2/2$. 

\section{Conclusion}
\label{sec:con}

The central limit for the product of positive free random variables is found to be the same as the sum of two independent free random variables with semicircle and uniform distribution, respectively.
The logarithm $\log Y$ of the central limit also have a simple moment-generating function $M_S(s)$ \eqn{logSmfun}.
The inverse Fourier transform of $M_S(i\omega)$ gives the probability density.


\end{document}